\documentclass[a4paper,10pt]{amsart}
\usepackage{amsmath,amsthm,amssymb,amsfonts,enumerate,color}
\input{amssym.def}
\input{amssym.tex}
\usepackage{verbatim}

\newtheorem{theorem}{Theorem}[section]
\newtheorem{lemma}[theorem]{Lemma}
\newtheorem{corollary}[theorem]{Corollary}
\newtheorem{proposition}[theorem]{Proposition}
\theoremstyle{definition}

\newtheorem{example}[theorem]{Example}

\newtheorem{remark}[theorem]{Remark}

\numberwithin{equation}{section}

\newcommand\beqn{\begin{equation}}
\newcommand\neqn{\end{equation}}
\newcommand{\al}{\alpha}   
\newcommand{\N}{\mathbb{N} } 
\newcommand{\R}{\mathbb{R} } 
\newcommand{\C}{\mathbb{C} } 

\newcommand{\HH}{\Bbb H}

\newcommand{\norm}[1]{\left\Vert#1\right\Vert}
\newcommand{\abs}[1]{\left\vert#1\right\vert}
\newcommand{\set}[1]{\left\{#1\right\}}

\newcommand{\Bc}{\mathcal{B}}

\allowdisplaybreaks

\begin{document}

\title[Extension problem and fractional operators]{Extension problem and fractional operators: \\ semigroups and wave equations}

\author[J. E. Gal\'e]{Jos\'{e} E. Gal\'{e}}
\address{Departamento de Matem\'{a}ticas e I. U. M. A.
Universidad de Zaragoza, 50009 Zaragoza, Spain}
\email{gale@unizar.es, pjmiana@unizar.es}

\author[P. J. Miana]{Pedro J. Miana}

\author[P. R. Stinga]{Pablo Ra\'ul Stinga}
\address{Department of Mathematics,
The University of Texas at Austin, 
1 University Station, C1200, Austin, 
TX 78712-1202, USA}
\email{stinga@math.utexas.edu}



\thanks{Research partially supported by Project MTM2010-16679, DGI-FEDER, of the MCYTS, Spain, and Project E-64, D.G. Arag\'on, Spain. The third author was partially supported by grant MTM2011-28149-C02-01 from Spanish Government.}

\begin{abstract}
We extend results of Caffarelli--Silvestre and Stinga--Torrea regarding a characterization of fractional powers of differential operators via an extension problem. Our results apply to generators of integrated families of operators, in particular to infinitesimal generators of bounded $C_0$ semigroups and operators with purely imaginary symbol. 
We give integral representations to the extension problem in terms of solutions to the heat equation and the wave equation.
\end{abstract}

\maketitle

\section{Introduction and main results}

Motivated by the study of regularity properties of solutions to nonlinear equations involving the fractional Laplacian 
$(-\Delta)^\sigma$, $0<\sigma<1$, L. Caffarelli and L. Silvestre looked for solutions of the Bessel-type differential equation
\begin{equation}\label{extension Laplacian}
\left\{
  \begin{array}{ll}
    \displaystyle u_{yy}+\tfrac{1-2\sigma}{y}u_y=-\Delta u, & y>0, \\
    u(0)=f,
  \end{array}
\right.
\end{equation}
where $f\in L^2(\mathbb{R}^n)$ and $u:[0,\infty)\to L^2(\mathbb{R}^n)$ is of class $C^2$.
They interpreted the function $u$ satisfying the equation above as the harmonic extension of $f$ to a fractional dimension $2-2\sigma$. Then $u$ is given by a Poisson type integral formula in terms of $f$. The remarkable property widely used in applications is that the nonlocal operator $(-\Delta)^\sigma$ acting on $f$ in the domain of  $\Delta$ is localized through $u$. More precisely, there exists a constant $c_\sigma<0$ such that
\begin{equation}\label{localization}
\lim_{y\to0^+}y^{1-2\sigma}u_y(y)=c_\sigma(-\Delta)^\sigma f,
\end{equation}
see \cite{Caffarelli-Silvestre}. For applications of the method see for instance \cite{Caffarelli-Salsa-Silvestre} and \cite{Caffarelli-Vasseur}.

Let us consider the {extension} problem \eqref{extension Laplacian} with $\Delta$ being replaced by a generic (closed) linear operator $A$. 
In \cite{Stinga-Torrea}, P. R. Stinga and J. L. Torrea gave a novel point of view to the extension problem by appealing to operator semigroup theory. 
This allowed them to get a general Poisson formula in terms of the semigroup $e^{tA}$ 
(that applies to many particular cases) and to get a Harnack inequality
for the fractional harmonic oscillator. 
Moreover, Bessel functions were used for the first time in \cite{Stinga-Torrea} to solve and analyze the extension problem. 
On the base of those tools they extended the Caffarelli--Silvestre theorem to a fairly general class 
of positive selfadjoint operators $D=-A$ having dense domains on $L^2$-spaces. Such a class includes, {for instance, the Laplacian in bounded domains,
elliptic Schr\"odinger operators $D=-\operatorname{div}(A\nabla)+V$ with suitable potentials $V$ and (weighted)
Laplace--Beltrami operators on (weighted) Riemannian manifolds. 
The approach to the extension problem given in \cite{Stinga-Torrea} has been applied in \cite{StingaZhang} to get Harnack's inequalities
for fractional powers of elliptic differential operators with measurable coefficients, and in \cite{RoncalStinga} 
to obtain a boundary Harnack inequality for the fractional Laplacian on the torus. Also some results of \cite{Stinga-Torrea}, like the general Poisson formula for the solution of the extension
problem, found applications in \cite{Ferrari-Franchi}, where Harnack's inequality for fractional sub-Laplacians in Carnot groups
are derived, and 
in \cite{Banica-Gonzalez-Saez}, where the fractional
Laplacians on the hyperbolic space and on some noncompact manifolds are studied.

Nevertheless, the results of \cite{Caffarelli-Silvestre} and \cite{Stinga-Torrea} do not seem to apply, totally or partially, to other important differential operators as, for example, those having purely imaginary symbol, in particular Schr\"odinger operators $A=i(\Delta+V)$, 
or those operators pointed out in Section \ref{Section:Applications} below. On the other hand,
besides its theoretical significance (see \cite{Yosida} for instance), fractional powers of operators appear in many concrete settings. 
{They occur,} for instance, when considering fractional kinetics and anomalous transport \cite{Zaslavsky}, in fractional quantum mechanics
\cite{Hu-Kallianpur} and fluid dynamics \cite{Caffarelli-Vasseur}, and also in mathematical finance when modeling with L\'evy processes
\cite{Caffarelli-Salsa-Silvestre}. Recently, the characterization (\ref{localization}) given in \cite{Caffarelli-Silvestre} has been used
to show that the fractional Laplacian $(-\Delta)^\sigma$ coincides with a certain conformally covariant operator $P_\sigma$ on
the hyperbolic space $\R^{n+1}_+$ from which $\R^n$ is to be seen as its boundary \cite{Chang-Gonzalez}.

Hence it seems sensible to characterize fractional operators $(-A)^\sigma$, in terms of solutions of local equations like in (\ref{localization}), for the widest possible class of operators $A$.
In the present paper, we show that the semigroup structure revealed in \cite{Stinga-Torrea}, 
as underlying problem \eqref{extension Laplacian}, 
can be subsumed in a more general framework which includes generators of $\alpha$-times 
integrated semigroups  and certain distribution semigroups. Working within that setting, we cover 
a wide range of important operators. Moreover, we are able to find new expressions for the solution to the extension problem that involves solutions to the associated wave equation.

To state the two main results we quickly introduce some notation and definitions. Let $\Bc(X)$ denote the Banach algebra of bounded operators on a Banach space $(X,\|\cdot\|)$. For $\alpha\geq0$, let $(T_\al(t))_{t\ge0}$ be a strongly continuous on $[0,\infty)$ family in ${\mathcal B}(X)$ such that 
$\sup_{t>0}t^{-\al}\Vert T_\al(t)\Vert<\infty$.

Suppose that there exists a (unique) closed and densely defined operator $A$ on $X$, and  
such that
$\lambda-A$ is invertible whenever $\Re\lambda>0$, with resolvent function
\begin{equation}\label{resolv}
(\lambda-A)^{-1}f=\lambda^\al\int_0^\infty e^{-\lambda t}T_\al(t)f\,dt,\quad \Re\lambda>0,~f\in X.
\end{equation}
Then we say that $(T_\alpha(t))_{t\ge0}$ is a {\it globally tempered} (or {\it tempered}, for short) {\it $\alpha$-times integrated semigroup} in $\Bc(X)$, and that $A$ is its 
{\it generator}.
Analogously, if $(T_\al(t))_{t\ge0}$ and $A$ are as before but, instead of (\ref {resolv}), they satisfy the relation
\begin{equation}\label{resolcos}
(\lambda^2-A)^{-1}f=\lambda^{\al-1}\int_0^\infty e^{-\lambda t}T_\al(t)f\,dt,\quad \Re\lambda>0,~f\in X,
\end{equation}
then we say that $(T_\alpha(t))_{t\ge0}$ is a {\it tempered $\alpha$-times integrated cosine family} 
in $\Bc(X)$ with {\it generator} $A$ (see \cite{ElMennaoui-Keyantuo}, for instance). Cosine families extend to $\R$ as even functions. 
Generators of integrated semigroups or integrated cosine families   
admit fractional powers $(-A)^{\sigma}$ in the Balakrishnan sense for every 
$ {0<\sigma<1}$, see \eqref{fractpow} below.

Integrated semigroups of integer order were first considered by W. Arendt in \cite{Arendt}, and of fractional order by M. Hieber in \cite{Hieber-Forum}. These semigroups are useful,
for instance, to obtain solutions of  
ill-posed abstract Cauchy problems, and even for non-densely defined operators. F. Neubrander also found applications
to second order Cauchy problems and non-autonomous equations in \cite{Neubrander}. 
Hieber gave applications to harmonic analysis and pseudodifferential operators, see subsection 2.4 in the present article.
For some detailed notes on these semigroups and  applications,  we refer the reader to  \cite[Section 3.17]{Arendt-Batty-Hieber-Neubrander}.

By $W^\alpha$ we denote the Weyl fractional derivative, see Section \ref{Section:Functions} for definitions. For real $\theta$ such that $0\le\theta\le\pi/2$, set
$S_\theta=\set{z\in\C:\abs{\arg{z}}<\theta}$ if $\theta>0$, and $S_0=(0,\infty)$.
Put $\C^+:=S_{\pi/2}=\{z\in\C:\Re z>0\}$.

\begin{theorem}\label{extension}
Fix $\alpha\ge0$. Let $A$ be the generator of a tempered $\alpha$-times integrated semigroup $(T_\alpha(t))_{t\ge0}\subseteq\Bc(X)$ {and} $0<\sigma<1$. Then a solution
$u:S_{\pi/4}\to\Bc(X)$ to the (vector-valued) differential equation
\begin{equation}\label{ExtSemig}
\left\{
  \begin{array}{ll}
    \displaystyle u''(z)+\tfrac{1-2\sigma}{z}\,u'(z)=-Au(z), & z\in S_{\pi/4}; \\
    \displaystyle\lim_{z\to0,z\in S_{\pi/4-\eta}}u(z)=f, & f\in X, 0<\eta\le\pi/4;
  \end{array}
\right.
\end{equation}
is given by
\begin{equation}\label{solution}
u(z)=\frac{z^{2\sigma}}{4^\sigma\Gamma(\sigma)}
\int_0^\infty W^\al\left(\frac{e^{-z^2/(4t)}}{t^{1+\sigma}}\right)T_\alpha(t)f\,dt,
\quad z\in S_{\pi/4}.
\end{equation}
This solution is uniformly bounded on subsectors of $S_{\pi/4}$; that is, there exists $C>0$ such that
$$
{\Vert u(z)\Vert<C\Vert f\Vert, \quad\hbox{for all}~z\in S_\eta,}
$$
for every $\eta$ such that $0\le\eta<\pi/4$.

Moreover, if $f\in {\mathcal D}(A)$, the domain of $A$, then
\begin{equation}\label{zwei}
\lim_{z\to0^+}\frac{{u(z)}-f}{z^{2\sigma}}=c_\sigma(-A)^\sigma f
={1\over 2\sigma}\lim_{{z\to0}}z^{1-2\sigma}{u'(z)},
\end{equation}
where $c_\sigma=4^{-\sigma}\Gamma(-\sigma)\Gamma(\sigma)^{-1}<0$ and both limits hold through proper subsectors of $S_{\pi/4}$.

Conversely, for $f\in {\mathcal D}(A)$ and $z\in S_{\pi/4}$, we have
\begin{equation}\label{ein}
u(z)=\lim_{\varepsilon\to0^+}\frac{1}{\Gamma(\sigma)}
\int_0^\infty W^\al\left(\frac{e^{-z^2/(4t)}e^{-\varepsilon t}}{t^{1-\sigma}}\right)
T_\al(t)(-A)^\sigma f\,dt,
\end{equation}
where the limit holds uniformly on compact subsets of $S_{\pi/4}$, and,
for $f\in {\mathcal D}(A)$ and $z\in \overline S_{\pi/4}$,
\begin{equation}\label{solfract}
\displaystyle u(z)=f+\frac{1}{\Gamma(\sigma)}
\int_0^\infty W^\al\left(\frac{e^{-z^2/(4t)}-1}{t^{1-\sigma}}\right)T_\al(t)(-A)^\sigma f\,dt.
\end{equation}
\end{theorem}

\begin{remark}
\normalfont
Formula (\ref{solfract}) extends the solution $u$ which {is initially} defined on the open sector $S_{\pi/4}$ to its closure, the closed sector $\overline S_{\pi/4}$.
\end{remark}

Each of the equalities in (\ref{zwei}) can be seen as a limit of localizations of the (in general) nonlocal operator $(-A)^\sigma$. The second one provides in this context the characterization of the fractional power
$(-A)^\sigma$ as a correspondence from the {Dirichlet boundary} condition $f\in\mathcal D(A)$ to the {Neumann}-type boundary condition
$\lim_{z\to0^+}z^{1-2\sigma}u'(z)$.

As mentioned before, we get a new formula for the solution $u$ in terms of the solution to the wave equation. We will use the notation $\mathcal C_\al(t)$ replacing $T_\alpha(t)$ to refer to cosine families.

\begin{theorem}\label{extensionCosine}
Let $A$ be the generator of a tempered $\al$-times integrated cosine function
$(\mathcal C_\al(s))_{s\in\R}\subseteq{\mathcal B}(X)$. Then the extension problem (\ref{ExtSemig}) for $A$ admits a solution $u(z)$ holomorphic in $z\in\C^+$ given by
\begin{equation}\label{solutionCosine}
u(z)=d_\sigma
\int_0^\infty W^\al\left({z^{2\sigma}\over(z^2+t^2)^{\sigma+1/2}}\right)
\, \mathcal C_\al(t)f\, dt,
\end{equation}
where ${d_\sigma=\frac{2\Gamma(\sigma+1/2)}{\sqrt\pi\Gamma(\sigma)}}$.

In addition, if $f\in\mathcal D(A)$ then  the solution $u(z)$ can be alternatively written for $z\in\C^+$ as
\begin{equation}\label{einCosine}
u(z)=f+\kappa_\sigma\int_0^\infty
W^\al\left((z^2+t^2)^{\sigma-{1/2}}-t^{2\sigma-1}\right)\,
{\mathcal C}_\al(t)(-A)^\sigma f\, dt,
\end{equation}
whenever $\sigma\not=1/2$, for
$\kappa_\sigma
={\frac{2\Gamma(1/2-\sigma)}{4^\sigma\sqrt{\pi}\Gamma(\sigma)}}$, or
$$
u(z)=f+{1\over\pi}
\int_0^\infty
W^\al \left({\rm Log}\left({t^2\over z^2+t^2}\right)\right)\,{\mathcal C}_\al(t)(-A)^{1/2}f\, dt,
$$
if $\sigma=1/2$. Here $\rm{Log}$ is the principal branch of the logarithm with argument in $[0,2\pi)$, and the fractional powers appearing in the integrals are referred to this logarithm.
\end{theorem}

In the classical case $A=\Delta$ on $\R^N$, the integrated cosine family in the statement above $v_\al(t)=\mathcal{C}_\al(t)f$ provides by derivation a mild solution to the wave equation with initial data $f$ and
null initial velocity
$$
\left\{
  \begin{array}{ll}
    v_\al''(t)-\Delta v_\al=0, & t>0, \\
    v_\al(0)=f,~v_\al'(0)=0.
  \end{array}
\right.
$$
The representation given in \eqref{solutionCosine}-\eqref{einCosine} of the solution to the extension problem via the solution to the wave equation is new. In the case of \eqref{extension} it reads
$$
u(x,y)=d_\sigma\int_0^\infty W^\al\left(\frac{y^{2\sigma}}
{(y^2+t^2)^{\sigma+1/2}}\right)\,v_\al(x,t)\,dt,
$$
for $\al>(N-1)\vert{1\over 2}-{1\over p}\vert$, since this is the value of $\al$ for $\Delta$ to generate an integrated cosine family in $L^p(\R^N)$, see \cite[Proposition 3.2]{ElMennaoui-Keyantuo}.

Uniformly bounded $C_0$-semigroups are globally tempered with $\alpha=0$. Then Theorem 
\ref{extension} extends in particular the results obtained in 
\cite[Sections~2~and~3]{Caffarelli-Silvestre} and \cite[Theorem~1.1]{Stinga-Torrea} to {\it all} infinitesimal
generators of bounded $C_0$-semigroups on Banach spaces. An important class of these semigroups are the so-called
bounded analytic semigroups $(T(z))_{z\in S_\theta}$ on sectors of the complex plane. It is known that a closed operator
$A$ is the infinitesimal generator of such a semigroup $T(z)$ if and only if it is the generator of a tempered integrated cosine family.
Hence, one can express the solution $u$ of (\ref{extension}) for such an operator $A$ by any of the formulas given in Theorem \ref{extension} and Theorem \ref{extensionCosine}.

As regards differential operators other than {those considered} in \cite{Caffarelli-Silvestre} and \cite{Stinga-Torrea}, and  to which Theorem \ref{extension} applies, one gets the differential operators $A$ on $L^p(\R^n)$ whose symbols are of the form $iq(\xi)$, $\xi\in\R$, where $q$ is a real elliptic polynomial or a ${C^{\infty}}$ homogeneous function on $\R^n\setminus\{0\}$ such that $q(\xi)=0$ implies $\xi=0$. Among them, there are the operators $\partial^{2n-1}/\partial x^{2n-1}$, $n\in\N$, in particular the operator $\partial^3/\partial x^3$ associated with the Korteweg--de Vries equation \cite{Hieber-Annalen}, or operators of the form $A=iL$ for suitable Laplacians $L$ on Riemannian manifolds or Lie groups \cite{Carron-Coulhon-Ouhabaz, Davies2, Davies3}. A more detailed list of examples are collected in
Section \ref{Section:Applications} . We also give there the proof of 
Theorem \ref{extensionCosine}.

Theorem \ref{extension} is proved in an abstract framework using Banach algebras and operator semigroup methods, so avoiding the dependency on the Fourier transform (and the Perron method) considered 
in \cite{Caffarelli-Silvestre} or on $L^2$ spectral methods (and Bessel functions) dealt with in 
\cite{Stinga-Torrea}. First we check a variant of equation (\ref{ExtSemig}) on some specific scalar functions; see Lemma \ref{Lem:b chica} (ii) and Lemma \ref{Lem:b grande} (iii). One of these functions lies in a suitable Sobolev algebra ${\mathcal T}^{(\al)}(t^\al)$, which is a convolution Banach algebra defined by means of {Weyl} fractional derivatives $W^\alpha$; see Section \ref{Section:Functions}. Then, via bounded algebra homomorphisms with domain
${\mathcal T}^{(\al)}(t^\al)$, the properties of these functions are transferred to operator semigroups.
On the way, a new formula for the Balakrishnan fractional power $(-A)^\sigma$ is given in Theorem 
\ref{Thm:fractional power} below in terms of the integrated semigroup $T_\alpha(t)$. Such a formula enters in the proof of 
Theorem \ref{extension}. 

Section \ref{Section:Functions} collects some preliminary results about functions in convolution Sobolev algebras. Section \ref{Section:Semigroups} contains elements of the integrated semigroups theory and the proof of Theorem \ref{Thm:fractional power}. 
 In Section \ref{Section:Proof} we prove  
Theorem \ref{extension}. Finally, we include in an appendix some observations on equation (\ref{ExtSemig}) in complex parameter 
$\sigma$.

Throughout the paper the letter $C$ denotes a constant that may change from line to line.

\section{Applications and related results}\label{Section:Applications}

\medskip
We start with the translation of Theorem \ref{extension} to the semigroup case

\subsection{Bounded $C_0$-semigroups}

As pointed out before, Theorem \ref{extension} is also valid for infinitesimal generators $A$ of uniformly bounded $C_0$-semigroups.
In this case such semigroups can be seen as the solution to the heat equation for $A$,
\begin{equation}\label{waveEq}
\left\{
  \begin{array}{ll}
    \displaystyle w'(s)=Aw(s), & \quad s{>}0; \\
    \displaystyle w(0)=f, & \quad f\in X.
  \end{array}
\right.
\end{equation}
The result reads as follows.

\begin{theorem}\label{GenSemig}
Let $A$ be the generator of a bounded $C_0$-semigroup $(T(t))_{t\ge0}\subseteq\Bc(X)$. Fix $0<\sigma<1$. Then a solution
$u:S_{\pi/4}\to\Bc(X)$ to the differential equation
(\ref{ExtSemig}) is given by
\begin{equation}\label{solutionSemi}
u(z)=\frac{z^{2\sigma}}{4^\sigma\Gamma(\sigma)}
\int_0^\infty \frac{e^{-z^2/(4t)}}{t^{1+\sigma}}T(t)f\,dt,
\quad z\in S_{\pi/4}.
\end{equation}
This solution is uniformly bounded on subsectors of $S_{\pi/4}$.

Moreover, if $f\in {\mathcal D}(A)$ then
\begin{equation}\label{zweiSemi}
\lim_{z\to0}\frac{{u(z)}-f}{z^{2\sigma}}=c_\sigma(-A)^\sigma f
={1\over 2\sigma}\lim_{z\to0}z^{1-2\sigma}u'(z),
\end{equation}
where $c_\sigma=4^{-\sigma}\Gamma(-\sigma)\Gamma(\sigma)^{-1}{<0}$ and both limits hold through proper subsectors of  
$S_{\pi/4}$.

Conversely, for $z\in S_{\pi/4}$ and $f\in {\mathcal D}(A)$, we have
\begin{equation}\label{einSemi}
\displaystyle u(z)=\lim_{\varepsilon\to0^+}
\int_0^\infty \frac{e^{-z^2/(4t)}e^{-\varepsilon t}}{t^{1-\sigma}}
T(t)(-A)^\sigma f\,\frac{dt}{\Gamma(\sigma)},
\end{equation}
where the limit holds uniformly on compact subsets of $S_{\pi/4}$, and, for $f\in {\mathcal D}(A)$ and
$z\in\overline S_{\pi/4}$,
\begin{equation}\label{solfractSemi}
\displaystyle u(z)=f+\frac{1}{\Gamma(\sigma)}
\int_0^\infty \frac{e^{-z^2/(4t)}-1}{t^{1-\sigma}}T(t)(-A)^\sigma f\,dt.
\end{equation}
\end{theorem}

\begin{remark} 
(i) Formulas (\ref{solutionSemi}), (\ref{zweiSemi}) appear in \cite[Theorem~1.1]{Stinga-Torrea}, 
where they are obtained for generators of semigroups arising as nonnegative selfadjoint operators acting  on $L^2$.
Theorem \ref{GenSemig} above is valid for {\it all} infinitesimal generators of bounded $C_0$-semigroups on Banach spaces, so it supplies the widest possible extension of  \cite[Theorem~1.1]{Stinga-Torrea} in the setting of operator semigroups; in particular it applies to (the suitable ones) differential operators on $L^p(\Omega)$, $1\le p\le\infty$. The proof of (\ref{solutionSemi}) and (\ref{zweiSemi}) given here, see Section \ref{Section:Proof}, are different from those given in \cite[Theorem~1.1]{Stinga-Torrea} and do not rely on either the spectral theorem or the Fourier transform.

(ii) Formulas (\ref{einSemi}) and (\ref{solfractSemi}) are new. The weight $e^{-\varepsilon t}$ and corresponding limit in (\ref{einSemi}) must be included because the vector-valued mapping
$t\mapsto e^{-z^2/4t}t^{\sigma-1}T(t)(-A)^\sigma f$ is not necessarily Bochner integrable on $(0,\infty)$. When the space $X$ is assumed to be reflexive and $A$ is a spectral operator of scalar type (see \cite{Dunford-Schwarz}  for the basic theory of these operators), the weight and limit can be removed from (\ref{einSemi}). Thus the solution to problem 
(\ref{ExtSemig}) in this case takes also the form
$$
\displaystyle u(z)= {1\over\Gamma(\sigma)}
\int_0^\infty{e^{-z^2/(4t)}\over t^{1-\sigma}}T(t)(-A)^\sigma f\,dt,\quad z\in S_{\pi/4},~f\in{\mathcal D}(A).
$$
The proof is a slight refinement ot that given in \cite{Stinga-Torrea} for nonnegative self-adjoint differential operators on $L^2$ spaces; see \cite[formula (1.7)]{Stinga-Torrea}.
\end{remark}
\subsubsection{The case of analytic semigroups}
Many interesting examples of $C_0$-semigroups turn out to be analytic and bounded in sectors.

Let $\theta\in(0,\pi/2]$. An analytic semigroup $(T(z))_{z\in S_\theta}\subseteq\Bc(X)$ on $S_\theta$ is said to be 
{\it bounded} if
$\lim_{z\to0}T(z)f=f$ and $\Vert T(z)\Vert$ is bounded on every proper subsector of $S_\theta$.
As a matter of fact, one has that if $A$ generates a uniformly bounded $C_0$-semigroup then the fractional power operator
$-(-A)^\rho$ is the infinitesimal generator of a bounded analytic semigroup for ${0<\rho<1}$; see \cite[p.~238]{Arendt-Batty-Hieber-Neubrander} or \cite[p.~263]{Yosida}.
We are particularly interested in semigroups analytic in $\C^+=S_{\pi/2}$.

Let assume that $(e^{-tL})_{t>0}\subseteq\mathcal{B}(X)$ is a semigroup with generator $-L$ and that there exists 
$\tau\ge0$ such that 
the semigroup $e^{-t(\tau+L)}$ admits an extension as a bounded analytic semigroup in $\C^+$. Note that this means that for some $\nu\ge0$ and a constant $C_{\nu,\tau}$ the property
\begin{equation}\label{growthalfa}
\Vert e^{-zL}\Vert\le C_{\nu,\tau}e^{\tau \Re z}
\left({\vert z\vert\over\Re z}\right)^\nu, \quad  z\in\C^+,
\end{equation}
is fulfilled. See \cite[Lemma~2]{Davies2}, \cite[Theorem~20~and~25]{Davies3}, \cite[Theorem~4.3]{Carron-Coulhon-Ouhabaz}.

There are many examples of semigroups on $X=L^p(\Omega)$ satisfying property (\ref{growthalfa}). (Here $\Omega$ is an open subset of $\R^n$ or of a more general manifold.) Next, we collect some of them. 

\begin{example}\label{examples}
({\it i}) The Laplacian $-\Delta$ or, more generally, Schr\"odinger operators $L:=-\Delta+V$ with $V$ in the Kato class, and magnetic Schr\"odinger operators acting initially on $L^2(\R^N)$ \cite[p.~178]{Davies2}, \cite[p.~303]{Carron-Coulhon-Ouhabaz}. ({\it ii}) Laplace-Beltrami operators acting on complete Riemannian manifolds of bounded geometry or with non-negative Ricci curvature \cite[p.~178]{Davies2}, \cite[pp.~299,~302]{Carron-Coulhon-Ouhabaz}.
({\it iii}) Uniformly elliptic differential operators, of second or higher order, with measurable coefficients acting initially on 
$L^2(\Omega)$ for $\Omega=\R^N$ or a compact Riemannian manifold without boundary \cite[p.~178]{Davies2}, 
\cite[pp.~159,~162,~165]{Davies3}. ({\it iv}) Bochner Laplacians acting on sections of vector bundles \cite[p.~178]{Davies2}. ({\it v}) Sub-Laplacians acting on Lie groups of polynomial growth and stratified nilpotent Lie groups; see \cite[p.~299]{Carron-Coulhon-Ouhabaz} and references therein.
({\it vi}) Distinguished sub-Laplacians acting on the Heisenberg group \cite{Muller}, \cite{Muller-Stein}, and generalized Heisenberg groups \cite[Proposition~2.4.1]{Randall}.
\end{example}

Bounded analytic semigroups are also closely related with wave equations. We treat this relation in our context later on.

\subsubsection{{The extension problem and the wave equation}}

Let $A$ be a closed and densely defined operator on a Banach space $X$. The wave equation for $A$ {with null initial velocity} is
\begin{equation}\label{waveEq}
\left\{
  \begin{array}{ll}
    \displaystyle w''(s)=Aw(s), & \quad s{>}0; \\
    \displaystyle w(0)=f,~w'(0)=0, & \quad f\in {\mathcal{D}(A)}.
  \end{array}
\right.
\end{equation}

Equation (\ref{waveEq}) admits a {unique} classical solution if and only if $A$ is the generator of a cosine function $\mathcal C_0(s)$, with $w(s)=\mathcal C_0(s)f$ \cite[Theorem~8.2]{Goldstein}. When $A$ generates a $\nu$-times
integrated cosine function $\mathcal C_\nu(s)$ with $\nu>0$ then it is still possible to find mild solutions to problem (\ref{waveEq}); see
\cite{Arendt-Kellerman}. 

It is implicitly contained in \cite{ElMennaoui-Keyantuo} and \cite{Keyantuo} altogether that a closed
operator $A$ generates a tempered integrated cosine function if and only if $A$ is the infinitesimal generator of a bounded analytic semigroup on $\C^+$. In particular, and more precisely, if 
 $A$ is the generator of an $\al$-times integrated cosine function
$\mathcal C_\al$ such that  $\Vert\mathcal C_\al(s)\Vert\le C\vert s\vert^\al$, $s\in\R$, then the operator $T(z)$ in $\mathcal B(X)$ defined by
\begin{equation}\label{CosToSemi}
T(z)f:=\int_0^\infty W^\al\left({e^{-s^2/4z}\over \sqrt{\pi z}}\right)\mathcal C_\al(s)f\, ds,
\quad f\in X, z\in\C^+,
\end{equation}
is a holomorphic semigroup generated by $A$ and such that
$\Vert T(z)\Vert\le C(\vert z\vert/\Re z)^{\al+(1/2)}$, for all $z\in\C^+$. (This fact is shown in \cite[p.~142]{Keyantuo} for positive integer $\al$; for any fractional $\al$ the proof is similar).
We next proceed to give the proof of {Theorem \ref{extensionCosine}}. 

\begin{proof}[Proof of Theorem \ref{extensionCosine}]
{Substituting} the expression of $T(t)$ given by (\ref{CosToSemi}) in the formula (\ref{solutionSemi}) of $u(y)$ with $y>0$ one gets
\begin{align*}
u(y)&=\frac{y^{2\sigma}}{4^\sigma\Gamma(\sigma)}
\int_0^\infty \frac{e^{-y^2/(4t)}}{t^{1+\sigma}}
\int_0^\infty W^\al\left({e^{{-s^2/(4t)}}\over \sqrt{\pi t}}\right)\mathcal C_\al(s)f\, ds\,dt,\\
&=\frac{y^{2\sigma}}{4^\sigma\sqrt\pi\Gamma(\sigma)}
\int_0^\infty\,W^\al\left(\int_0^\infty{e^{-(y^2+s^2)/(4t)}\over t^{\sigma+1/2}}{dt\over t}\right)\,
\mathcal C_\al(s)f\,ds\\
&=\frac{y^{2\sigma}}{4^\sigma\sqrt\pi\Gamma(\sigma)}
\int_0^\infty\,W^\al\left(\int_0^\infty e^{-r}\left({4r\over y^2+s^2}\right)^{\sigma+1/2}{dr\over r}\right)\mathcal C_\al(s)f\,ds\\
&={2\Gamma(\sigma+1/2)\over\sqrt\pi\Gamma(\sigma)}
\int_0^\infty\,W^\al\left({y^{2\sigma}\over(y^2+s^2)^{\sigma+1/2}}\right)
\mathcal C_\al(s)f\,ds.
\end{align*}

{Let us} now assume that $f\in\mathcal D(A)$. Putting (\ref{CosToSemi}) in (\ref{solfractSemi}) we have that, for every $y>0$
and $\mu_\sigma:=(\sqrt{\pi}\Gamma(\sigma))^{-1}$,
$$
u(y)-f=\mu_\sigma
\int_0^\infty W^\al\left(\int_0^\infty {e^{-s^2/{(4t)}}(e^{-y^2/(4t)}-1)\over t^{1/2-\sigma}}
{dt\over t}\right) \mathcal C_\al(s)(-A)^\sigma f\, ds.
$$

In order to {compute} $\displaystyle F(s,y):=\int_0^\infty e^{-s^2/(4t)}(e^{-y^2/(4t)}-1)t^{\sigma-3/2}dt$, note that this integral is analytic in $\Re \sigma<3/2$. For a while, suppose that $\sigma$ is such that $0<\sigma<1/2$. Then, for $s,y>0$,
\begin{align*}
F(s,y)&=\int_0^\infty{e^{-(s^2+y^2)/(4t)}\over t^{1/2-\sigma}}{dt\over t}-
\int_0^\infty{e^{-s^2/(4t)}\over t^{1/2-\sigma}}{dt\over t}\\
&=
4^{(1/2)-\sigma}\Gamma(1/2-\sigma)
\left((s^2+y^2)^{\sigma-1/2}-(s^2)^{\sigma-1/2}\right).
\end{align*}
If $\sigma=1/2$, $a=(s^2+y^2)/4$ and $b=s^2/4$ we get by partial derivation in $a$
and $b$ under the integral that
\begin{align*}
F(s,y)&=\int_0^\infty(e^{-a/t}-e^{-b/t}){dt\over t}
=\int_0^\infty(e^{-ar}-e^{-br})\,{{dr\over r}}\\
&=\log(b/a)=\log\left({s^2\over y^2+s^2}\right), \quad\hbox{for all}~s,y>0.
\end{align*}
This proves the theorem.
\end{proof}

\subsubsection{Boundary values as integrated groups}

Boundary values of analytic semigroups in $\C^+$ give rise to integrated groups or semigroups. The following theorem is \cite[Corollary~5.3]{ElMennaoui}, see also \cite[Theorem~2.3]{ElMennaoui-Keyantuo}.

 \begin{theorem}\label{bdryValues}
Let $\nu\ge0$ and let $H$ be the infinitesimal generator of an analytic semigroup $T(z)$ in $\C^+$. The following assertions are equivalent.

(a) For all $\al>\nu$ there exists $\tau\ge0$ such that the semigroup $T(z)$ satisfies estimate (\ref{growthalfa}) with some constant $C_{\al,\tau}$.

(b) For all $\al>\nu$ the operator $iH$ generates a globally tempered $\al$-times integrated group $T_\al(it)$, $t\in\R$, of growth $\al$.
\end{theorem}

According to this result, Theorem \ref{extension} applies to $A=-iL$ for $L$ any of the operators included {in Examples \ref{examples}. These} cases deserve a particular digression, which is done in the {next} subsection.

\subsection{Operators $iH$}

Let $H$ be the infinitesimal generator of an analytic semigroup in $\C^+$ satisfying the estimate (\ref{growthalfa}) and let $\al>\nu$. 
Then, by Theorem \ref{bdryValues}, $iH$ generates an $\al$-times integrated group $T_\al(it)$ such that 
$\Vert T_\al(it)\Vert\le C \vert t\vert^\al$, $t\in\R$.

Let us consider the extension problem (\ref{ExtSemig}) for $A=iH$ and $\sigma=1/2$: $u_{yy}=-iHu$, $u(0)=f$. It is very simple to find a solution of this problem. Actually, it is enough to take the solution $v(z)=e^{-z\sqrt{-H}}f$ $(z\in\C^+)$ to the problem (\ref{ExtSemig}) for $A=\sqrt{-H}$, $\sigma=1/2$, and make $u(y):=v((\sqrt 2/2)(1+i)y)$, $y\in\R$. Clearly, the existence of this solution is possible because the semigroup
$e^{-t\sqrt{-H}}$, $t>0$, admits an analytic extension to $\C^+$. 

Let now $\sigma\not=1/2$. Once again the function $u(y):=v((\sqrt 2/2)(1+i)y)$ is the {\it formal} solution in 
Theorem \ref{extension} of (\ref{ExtSemig}) for the integrated group $T_\al(iy)$ in this case, if $v$ is the solution in 
Theorem \ref{GenSemig} of (\ref{extension}) for the semigroup $e^{-t\sqrt{-H}}$. However, to give a sense to the function above one needs that $e^{-t\sqrt{-H}}$ admits an extension up to the boundary of $S_{\pi/4}$ at least, and that 
$e^{-t\sqrt{-H}}$ is bounded on $(0,\infty)$ (for which $\tau$ must be zero in (\ref{growthalfa})).
Assuming that this is fulfilled we get the following corollary.

\begin{corollary}
Let $H$ be the generator of an analytic semigroup $e^{zH}$ satisfying, for some $\al,\tau\ge0$,
$$
\Vert e^{zH}\Vert\le C e^{\tau \Re z}
\left({\vert z\vert\over\Re z}\right)^\al, \quad \hbox{for all}~z\in\C^+.
$$
Let $T_\al(it)$ be the tempered $\al$-times integrated group generated by $iH$.
Let $\sigma\not=1/2$. Then {\rm(i)} the problem
\begin{equation}\label{iextension}
\left\{
  \begin{array}{ll}
    \displaystyle u_{yy}+\tfrac{1-2\sigma}{y}\,u_y=-iH u; \\
    u(0)=f \quad f\in X ;
    \end{array}
\right.
\end{equation}
has a solution given by formula (\ref{solution}), and also by (\ref{ein}) or (\ref{solfract}) if
$f\in \mathcal D(H)$, with $T_\al(it)$, $t\in\R$, in the integrals;

{\rm(ii)} in the case when $\tau=0$ and $f\in \mathcal D(\sqrt {-H})$ that solution $u$ coincides with
$u(y)=v((\sqrt2/2)(1+i)y)$, $y\in\R$, where $v$ is the solution to the problem
\begin{equation}\label{sqrt}
\left\{
  \begin{array}{ll}
    \displaystyle v_{yy}+\tfrac{1-2\sigma}{y}~v_y=\sqrt{-H} v; \\
    u(0)=f \quad f\in\mathcal D(\sqrt {-H}),
    \end{array}
\right.
\end{equation}
given by formulae (\ref{einSemi}) or (\ref{solfractSemi}) applied to the semigroup
$e^{-z\sqrt{-H}}$.
\end{corollary}

The case $\sigma=1/2$ is trivial, as seen prior to the theorem.

\begin{remark}
Since an operator $H$ as in the corollary above generates a tempered cosine family $\mathcal C_\beta(s)$, the solution $u$ to equation (\ref{iextension}) can be also expressed as $u(y)=v((\sqrt2/2)(1+i)y)$, for $y\in\R$, where $v$ is the solution of (\ref{sqrt}) given by the integral formulas (\ref{solutionCosine}), for any $f\in X$, and (\ref{einCosine}) associated to $\mathcal C_\beta(s)$.
\end{remark}

Thus the corollary applies to all the examples $H=-L$ listed in Example \ref{examples}. 
Among them we have the Schr\"odinger operator $i\Delta$ which is related to the Schr\"odinger 
equation of quantum mechanics \cite{Hu-Kallianpur}. In this case one can obtain solutions to the 
degenerate equation with complex coefficients $i\Delta u+\tfrac{1-2\sigma}{y}\,u_y+u_{yy}=0=\operatorname{div}
  \left(M\cdot\nabla u\right)$ in $\R^n\times(0,\infty)$, $u(x,0)=f(x)$, 
with {the $(n+1)\times(n+1)$ diagonal matrix $M=\operatorname{diag}(i,\ldots,i,y^{1-2\sigma})$, which, by Theorem \ref{extension},} satisfy the $L^p$ estimate
$$
\sup_{y>0}\norm{u(y)}_{L^p(\R^n)}\leq C\norm{f}_{L^p(\R^n)}.
$$
Note that the coefficients of the equation above are not bounded, so the theory of elliptic equations with complex {bounded} measurable coefficients does not apply to it.

Fractional Schr\"odinger equations of the form $v_t=(-i\Delta)^\sigma v$, $v(x,0)=f(x)$, can be studied as limit cases when
$z\to i$ of equations $v_t=(-z\Delta)^\sigma v$ for $\Re z>0$, see \cite[p.~283]{Hu-Kallianpur}. By Theorem \ref{extension} such an equation can be written in this equivalent form for a function $u=u(x,t,y)$:
$$
\left\{
  \begin{array}{ll}
  \displaystyle  i\Delta u+\tfrac{1-2\sigma}{y}\,u_y+u_{yy}=0, & \hbox{in}~\R^n\times\R\times(0,\infty);
    \\
   \displaystyle   u_t(x,t,0)+\tfrac{1}{2\sigma c_\sigma}\lim_{y\to0^+}y^{1-2\sigma}u_y(x,t,y)=0, & \hbox{in}~\R^n\times\R.
  \end{array}
\right.
$$

Note that the operators $H$ considered in the preceding lines have all real spectrum. In other words, operators $iH$ are particular cases of operators with pure imaginary symbols. We discuss this more general class of operators in the next subsection.

\subsection{Operators with pure imaginary symbol}

Conditions for operators $A$ with purely imaginary symbols to generate integrated semigroups have been thoroughly studied by M. Hieber in a series of papers; see \cite{Hieber-Annalen,Hieber-Forum,Hieber-JMathAnal,Hieber-Zeits,Hieber-TrAMS}, or \cite[Part~C]{Arendt-Batty-Hieber-Neubrander}. As a consequence of such an investigation we can apply our Theorem \ref{extension} to the following classes.

  (1) Operators $A$ for which their symbols ${\widehat A}$ are given by ${\widehat A}(\xi)=iq(\xi)$, $\xi\in\R^N$, where $q$ is a real elliptic homogeneous polynomial or a real homogeneous ${C^\infty}$ function on $\R^N\setminus\{0\}$ such that $\xi=0$ if $q(\xi)=0$. They generate globally tempered  integrated semigroups; see \cite[Theorem~4.2]{Hieber-Annalen}.
This result provides us with a lot of examples. In particular we wish emphasize on the linear  operator 
$\partial^3_x=\partial^3/\partial x^3$ associated to the general Korteweg--de Vries equation,
\begin{equation}\label{KdeV}
\partial_t u=-\partial_x^3 u+\partial_xP(u),\quad(x,t)\in\R^2,
\end{equation}
where $P(u)$ is some polynomial ($P(u)=6u$ for the usual Korteweg--de Vries equation). It is known that solutions of equation  (\ref{KdeV}) can be constructed from solutions of the equation
\begin{equation}\label{KdeVlinear}
\partial_t u=-\partial_x^3 u,\quad(x,t)\in\R^2,
\end{equation}
see \cite{Taflin} and references therein. (In turn, a solution to the linearized KdV equation $v_t=-v_{xxx}$ is given by 
$v(x,t)=t^{-1/3} w(x/(3t)^{1/3})$ where $w$ is a solution of the homogeneous Airy equation $w_{xx}-xw=0$ 
\cite[pp.~129,~130]{Vallee-Soares}).
By the aforementioned results of Hieber we have that the operator $\partial_x^3$ is 
the generator of a tempered $\al$-times integrated semigroup on $L^p(\R)$ for $\al> \vert 1/p-1/2\vert$;
see \cite[p.~15]{Hieber-Annalen}. Hence Theorem \ref{extension} holds also true for that operator. Therefore, the problem
        $$
        \left\{
          \begin{array}{ll}
             \tfrac{1-2\sigma}{y}\,u_y+u_{yy}=u_{xxx}, & \hbox{in}~\R\times\R\times(0,\infty); \\
          \displaystyle u_t(x,t,0)-\frac{1}{2\sigma c_\sigma}\lim_{y\to0^+}y^{1-2\sigma}u_y(x,t,y)=0, & \hbox{in}~\R\times\R;
          \end{array}
        \right.
        $$
        is equivalent to the fractional Korteweg--de Vries equation
        $$
        \left\{
          \begin{array}{ll}
          w_t=(-\partial_{xxx})^\sigma w, & \hbox{in}~\R^2,~0<\sigma<1; \\
          w(x,0)=f(x), & \hbox{on}~\R.
          \end{array}
        \right.
        $$
        with $w(x,t)\equiv u(x,t,0)$.
    
    (2) Operators $A=P(D)$ on $L^p(\R^N)^m$, with $1<p<\infty$, arising in symmetric hyperbolic systems like Maxwell`s equations, elastic waves in homogeneous medium, or neutrino equation. If the symbol of these operators is homogeneous then $P(D)$ generates a globally tempered integrated semigroup. See \cite[Corollary~3.2~and~Section~4]{Hieber-JMathAnal}.
    
    (3) Operators with symbols which are homogeneous coercive polynomials on $\R^N$ and have spectrum 
    in a half-plane \cite[Theorem~5.7]{Hieber-Zeits}.
    
    (4) Pseudodifferential operators on $L^p$ with spectrum within a half-plane and homogeneous symbol \cite[Theorem~4.6~and~Remark~4.7~b)]{Hieber-TrAMS}.

\section{Functions in convolution Sobolev algebras}\label{Section:Functions}

For $\alpha>0$ we define the Banach space of Sobolev type $\mathcal{T}^{(\alpha)}(t^{\alpha})$ as the completion of the test function space $C_c^\infty[0,\infty)$ in the norm
$$
\norm{\varphi}_{(\alpha)}:={1\over\Gamma(\alpha+1)}\int_0^\infty|W^{\alpha}\varphi(t)|t^\alpha\,dt,
\quad \varphi\in C_c^\infty[0,\infty).
$$
Here $W^\alpha \varphi$ is the {Weyl} fractional derivative of $\varphi$ of order $\alpha$ given by the formulae
$$
W^{-\beta}\varphi(s):={1\over \Gamma(\beta)}\int_s^\infty(t-s)^{\beta-1}\varphi(t)\,dt,
\quad s\geq 0,
$$
and
$$
W^\alpha \varphi(s):=(-1)^n\frac{d^n}{ds^n}\,W^{-\beta}\varphi(s), \quad s\geq 0,
$$
where $\beta=n-\alpha$, with $n=[\alpha]+1$.
In fact, the mapping from $C^\infty_c[0,\infty)$ into itself given by
$\varphi(t)\longmapsto t^\al W^\al \varphi(t)$ extends to a Banach isomorphism from $\mathcal{T}^{(\alpha)}(t^{\alpha})$ into $L^1(\R^+)$, so that a function
$\varphi$ belongs to $\mathcal{T}^{(\alpha)}(t^{\alpha})$
if and only if there exists $g\in L^1(\R^+,t^\alpha\,dt)$ such that
$$
\varphi(s)={1\over\Gamma(\alpha)}\int_s^\infty (t-s)^{\alpha-1}g(t)\,dt,\quad s>0.
$$
Then we put $g:=W^\alpha \varphi$. Note that $W^n=(-1)^n(d^n/dt^n)$ for every
$n\in\N\cup\{0\}$
and that $W^{\alpha+\beta}\varphi=W^\alpha(W^\beta \varphi)$ whenever the composition has sense.
Also, inclusion mappings 
$C_c^\infty[0,\infty)\hookrightarrow\mathcal{S}[0,\infty)
\hookrightarrow\mathcal{T}^{(\beta)}(t^{\beta}) \hookrightarrow\mathcal{T}^{(\alpha)}(t^{\alpha})
\hookrightarrow\mathcal{T}^{(0)}(t^{0})\equiv L^1(\mathbb{R}^+)$,
for $\beta>\alpha>0$, where $\mathcal{S}[0,\infty)$ is the restriction of the Schwartz's class $\mathcal{S}(\R)$ to 
$[0,\infty)$, are continuous. Here, $C_c^\infty[0,\infty)$ and $\mathcal{S}[0,\infty)$ are taken endowed with their respective usual topologies.

Furthermore, $\mathcal{T}^{(\alpha)}(t^{\alpha})$ is a Banach algebra with product defined by the convolution on $\R^+$, which in this paper means that the convolution is jointly continuous with respect to the norm $\Vert\cdot\Vert_{(\alpha)}$. See \cite{Gale-Miana-JFA, Miana-Forum} for these and some other properties of the Banach algebras $\mathcal{T}^{(\alpha)}(t^{\alpha})$. We will refer to these algebras here as (convolution) Sobolev algebras.

In the remainder of this section we give a list of properties of some functions in the Sobolev algebras. Such functions are used in the proof of Theorem \ref{extension}, in Section \ref{Section:Proof} below.

Given $\sigma>0$, define
\begin{equation}\label{bdef}
b^{\sigma,z}(t):={z^{2\sigma}\over4^\sigma \Gamma(\sigma)}\,{e^{-z^2/(4t)}\over t^{1+\sigma}},\quad t>0,~z\in S_{\pi/4}.
\end{equation}
Obviously, $b^{\sigma,z}\in L^1(\R^+)$ for all ${\sigma>0},~z\in S_{\pi/4}$. Next we point out several identities concerning derivatives of the function $b^{\sigma,z}$. In particular we will see that $b^{\sigma,z}$ belongs to the Sobolev algebras $\mathcal{T}^{(\alpha)}(t^\alpha)$ for any order of derivation $\alpha$.

\begin{lemma}\label{Lem:b chica}
Let $z\in S_{\pi/4}$, ${\sigma>0}$ and $n\in\N$.
\begin{itemize}
\item[\rm{(i)}] $\displaystyle{\partial^n\over\partial t^n}\,b^{\sigma,z}(t)
=\left(\sum_{j=0}^nd_{j,n}^\sigma{z^{2j}\over 4^jt^{j+n}}\right)b^{\sigma,z}(t)$,
with $d_{j,n}^\sigma\in\R$. More precisely, for $n=1$,
$$
{\partial\over \partial t}\,b^{\sigma,z}(t)=\left({z^2\over 4t^2}
-{1+\sigma\over t}\right)b^{\sigma,z}(t).
$$
\item[{(ii)}] $\displaystyle{\partial^n\over\partial z^n}\,b^{\sigma,z}(t)
=\left(\sum_{j=0}^nc_{j,n}^\sigma{z^{2j-n}\over t^j}\right)b^{\sigma,z}(t)$,
with $c_{j,n}^\sigma\in\R$. More precisely, for $n=1,2$,
$$
{\partial\over \partial z}\,b^{\sigma,z}(t)
=\left({2\sigma\over z}-{z\over 2t}\right)b^{\sigma,z}(t)
$$
and
$$
{\partial^2\over\partial z^2}\,b^{\sigma,z}(t)
=\left({2\sigma(2\sigma-1)\over z^2}-{4\sigma+1\over2t}+{z^2\over4t^2}\right)b^{\sigma,z}(t).
$$
\item[{(iii)}] The function $b^{\sigma,z}$ satisfies the differential equations
$$
2t{\partial\over\partial t}\,u(z,t)+z{\partial\over\partial z}\,u(z,t)=-2u(z,t),
$$
and
$$
{\partial^2\over\partial z^2}\,u(z,t)+{1-2\sigma\over z}\,{\partial\over\partial z}\,u(z,t)={\partial\over\partial t}\,u(z,t),
$$
with the boundary conditions
$\displaystyle\lim_{t\to0^+}u(z,t)=\lim_{t\to\infty}u(z,t)=0$.
\end{itemize}
\end{lemma}

\begin{proof}
We just compute $\displaystyle {\partial^n\over\partial t^n}\,b^{\sigma,z}$. The remainder of the proof is left to the interested reader. Put $a=-z^2/4$. For every $m\in\N$ and $p=1,\ldots,m$ there are constants ${c_{p,m}\in\R}$, such that
\begin{equation}\label{EatPre}
(e^{at^{-1}})^{(m)}(t)=(-1)^m\sum_{p=1}^mc_{p,m}(a/t)^pt^{-m}e^{at^{-1}}
=p_m(a/t)t^{-m}e^{at^{-1}},
\end{equation}
where $p_m$ is a polynomial of degree $m$, with $p_m(0)=0$. Hence, for all $n\in\N$ and $\sigma>0$,
\begin{equation}\label{Eat}
(e^{at^{-1}}t^{-(\sigma+1)})^{(n)}(t)=(-1)^n\sum_{m=0}^n c_{m,n,\sigma} p_m(a/t)t^{-(\sigma+n+1)}e^{at^{-1}},
\end{equation}
with $p_0\equiv1$, where $c_{m,n,\sigma}
=\displaystyle{n\choose m}{\sigma+n-m\choose n-m}(n-m)!$. By reordering terms, we get
\begin{equation}\label{EatDes}
(e^{at^{-1}}t^{-(\sigma+1)})^{(n)}(t)
=\sum_{k=0}^nd_{k,n}^\sigma(a/t)^k t^{-(\sigma+n+1)}e^{at^{-1}}.
\end{equation}
Therefore, $\displaystyle{\partial^n\over\partial t^n}\,b^{\sigma,z}(t)
=\left(\sum_{j=0}^n d_{j,n}^\sigma{z^{2j}\over 4^jt^{j+n}}\right)b^{\sigma,z}(t)$,
for some real constants $d_{j,n}^\sigma$, as claimed in the statement.
\end{proof}

Put $e_\varepsilon(t):=e^{-\varepsilon t}$ for all $\varepsilon,t>0$.

\begin{proposition}\label{Prop:b chica}
Let $(b^{\sigma,z})_{z\in S_{\pi/4}}$ be as in \eqref{bdef}.
\begin{itemize}
\item[\rm{(i)}] For every $N\in\N$, we have
$(b^{\sigma,z})_{z\in S_{\pi/4}}\subseteq\mathcal{T}^{(N)}(t^N)$ with
$$
\norm{b^{\sigma,z}}_{(N)}\le C_{N,\sigma}\left(\abs{z}^2\over\Re(z^2)\right)^{N+\sigma},\quad z\in S_{\pi/4}.
$$
Moreover, the map $z\longmapsto b^{\sigma,z},\,
S_{\pi/4}\to\mathcal{T}^{(N)}(t^{N})$ is analytic. In consequence, $(b^{\sigma,z})_{z\in S_{\pi/4}}$ is an analytic family of functions in $\mathcal{T}^{(\alpha)}(t^\alpha)$, for all $\alpha>0$.
\item[\rm{(ii)}]  For every $\alpha>0$, the equality
$$
{\partial^2\over \partial z^2}\,b^{\sigma,z}
+{1-2\sigma\over z}\,{\partial\over\partial z}\,b^{\sigma,z}
={\partial\over\partial t}\,b^{\sigma,z},\quad z\in S_{\pi/4},
$$
holds in the Banach algebra $\mathcal{T}^{(\alpha)}(t^{\alpha})$.
\item[\rm{(iii)}] For every $\varepsilon>0$ and $\alpha>0$ we have  $b^{\sigma,z}e_\varepsilon\in\mathcal{T}^{(\alpha)}(t^{\alpha})$ and,
if $0<\sigma<1$ then $\lim_{\varepsilon\to0^+}b^{\sigma,z}e_\varepsilon=b^{\sigma}_z$
in the norm of $\mathcal{T}^{(\alpha)}(t^{\alpha})$,
and uniformly on compact subsets of $S_{\pi/4}$.
\end{itemize}
\end{proposition}

\begin{proof}
(i) Take $z\in S_{\pi/4}$, $\sigma>0$ and $N\in\N$. By Lemma
\ref{Lem:b chica}\textit{(i)},
\begin{align*}
&\norm{b^{\sigma,z}}_{(N)} 
\le \sum_{j=0}^N|d_{j,N}^\sigma|4^{-(j+\sigma)}{\abs{z}^{2(\sigma+j)}
\over N!\Gamma(\sigma)}\int_0^\infty t^{-(1+j+\sigma)}|e^{-z^2/(4t)}|\,dt \\
&= \sum_{j=0}^N{2|d_{j,N}^\sigma|\abs{z}^{2(\sigma+j)}
\over N!\Gamma(\sigma)\Re(z^2)^{\sigma+j}}\int_0^\infty u^{2(\sigma+j)-1}e^{-u^2}du
\le C_{N,\sigma}\left(\abs{z}^2\over\Re(z^2)\right)^{N+\sigma},
\end{align*}
where we have used the change of variable  $u^2=\Re(z^2)/(4t)$ in the equality. Now we apply that $\mathcal{T}^{(N)}(t^{N})\hookrightarrow\mathcal{T}^{(\alpha)}(t^{\alpha})$ if $N\ge\al$ to conclude that $b^{\sigma,z}\in\mathcal{T}^{(\alpha)}(t^{\alpha})$.
The analyticity of the map $z\longmapsto b^{\sigma,z}$,
$z\in S_{\pi/4}\to\mathcal{T}^{(\al)}(t^{\al})$, follows from its continuity (use the dominated convergence theorem) and Morera's theorem.

(ii) The assertion holds because, for $z\in S_{\pi/4}$, the functions $\displaystyle{\partial^2\over\partial z^2}\,b^{\sigma}_z$,
$\displaystyle{\partial\over\partial z}\,b^{\sigma}_z$ and
$\displaystyle{\partial\over\partial t}\,b^{\sigma}_z$ belong to $\mathcal{T}^{(\alpha)}(t^{\alpha})$ for every $\alpha\ge0$ (see Lemma \ref{Lem:b chica}).

(iii) Write $a=-z^2/4$. It follows by \eqref{Eat} that, for $n=0,1,\ldots,N$ and $t>0$,
\begin{align*}
    \vert(e^{at^{-1}}t^{-(\sigma+1)})^{(n)}(t)\vert &\le \sum_{k=0}^n\vert d_{k,n}^\sigma\vert \vert a\vert^k t^{-k} e^{(\Re a/2)t^{-1}}t^{-(\sigma+n+1)} e^{(\Re a/2)t^{-1}} \\
     &\le \sum_{k=0}^n \vert d_{k,n}^\sigma\vert \vert a\vert^k\left(2k\over\vert \Re a\vert\right)^k e^{-k} t^{-(\sigma+n+1)} e^{(\Re a/2)t^{-1}}.
\end{align*}
Hence, if $K$ is a compact subset of $S_{\pi/4}$ then there exist constants $C_{\sigma,K},M_K>0$ such that
$$\vert(b^{\sigma,z})^{(n)}(t)\vert
\le C_{\sigma,K} t^{-(\sigma+n+1)}e^{-M_K t^{-1}},$$
for $t>0$, $z\in K$. Then, if $0<\sigma<1$, for every $z\in K$ we have
\begin{align*}
\Vert &b^{\sigma,z} e_\varepsilon-b^{\sigma,z}\Vert_{(N)} \le \sum_{n=0}^{N-1}
\int_0^\infty\vert(b^{\sigma,z})^{(n)}(t)\vert
\varepsilon^{N-n}e^{-\varepsilon t}t^Ndt \\
&\qquad\qquad\qquad\qquad+\sum_{n=0}^{N-1}\int_0^\infty\vert(b^{\sigma,z})^{(N)}(t)\vert(1-e^{-\varepsilon t})t^Ndt \\
&\le C_{\sigma,K}
\left[\sum_{n=0}^{N-1}\varepsilon^{N-n}
\int_0^\infty t^{N-n}e^{-\varepsilon t}\frac{dt}{t^{1+\sigma}}
+\int_0^\infty e^{-M_K/t}\left(1-e^{-\varepsilon t}\right){dt\over t^{\sigma+1}}\right] \\
&= C_{\sigma,K}\left[\sum_{n=0}^{N-1}\Gamma(N-\sigma-n)\varepsilon^\sigma
+\int_0^\infty e^{-M_K/t}\left(1-e^{-\varepsilon t}\right){dt\over t^{\sigma+1}}\right].
\end{align*}
From here, by using the dominated convergence theorem in the last integral we obtain that
$$\displaystyle\sup_{z\in K}\norm{b^{\sigma,z} e_\varepsilon-b^{\sigma,z}}_{(N)}\to0,$$
as $\varepsilon\to0^+$.
\end{proof}

For $t,\sigma>0$ and $z\in S_{\pi/4}$, define
\begin{equation}\label{B grande}
B^{\sigma,z}(t):=4^\sigma\left(\frac{t}{z}\right)^{2\sigma}b^{\sigma,z}(t)
={1\over\Gamma(\sigma)}\,{e^{-z^2/(4t)}\over t^{1-\sigma}}.
\end{equation}
It is readily seen that, for $0<\sigma<1$, 
\begin{equation}\label{derB}
z^{1-2\sigma}{\partial\over\partial z}\,B^{\sigma,z}
={\sigma\Gamma(-\sigma)\over2^{2\sigma-1}\Gamma(\sigma)}\,b^{1-\sigma,z},
\quad z\in S_{\pi/4}.
\end{equation}

Besides the preceding relationships  between the functions $b^{\sigma,z}$ and $B^{\sigma,z}$, we have the following identities involving $B^{\sigma,z}$, whose proof is left to prospective readers.

\begin{lemma}\label{Lem:b grande}
Let $z\in S_{\pi/4}$, $t>0$ and $n\in\N$.
\begin{itemize}
\item[{(i)}] $\displaystyle {\partial^n\over\partial t^n}\,B^{\sigma,z}(t)
=\left(\sum_{j=0}^nk_{j,n}^\sigma{z^{2j}\over4^jt^{j+n}}\right)B^{\sigma,z}(t)$,
with constants $k_{j,n}^\sigma\in\R$ and $k_{0,n}^\sigma=(\sigma-1)\cdots(\sigma-n)$.
In particular,
$$
{\partial\over\partial t}\,B^{\sigma,z}(t)
=\left({z^2\over4t^2}-{1-\sigma\over t}\right)B^{\sigma,z}(t).
$$
\item[{(ii)}]
$\displaystyle
{\partial\over\partial z}\,B^{\sigma,z}(t)
=-{z\over2t}\,B^{\sigma,z}(t)$, and
$$
{\partial^2\over\partial z^2}\,B^{\sigma,z}(t)
=\left({z^2\over 4t^2}-{1\over 2t}\right)B^{\sigma,z}(t).
$$
\item[{(iii)}] The function $B^{\sigma,z}$ satisfies the differential equations
$$
2t{\partial\over\partial t}\,u(z,t)+z{\partial\over\partial z}\,u(z,t)=-2(1-\sigma)u(z,t),
$$
and
$$
{\partial^2\over\partial z^2}\,u(z,t)+{1-2\sigma\over z}\,{\partial\over\partial z}u(z,t)
={\partial\over\partial t}\,u(z,t),
$$
with the boundary conditions $\displaystyle\lim_{t\to 0^+}u(z,t)
=\lim_{t\to\infty}u(z,t)=0$.
\end{itemize}
\end{lemma}

Let $h^\sigma$ denote the Heaviside-type function defined by
$$
h^\sigma(t):={1\over\Gamma(\sigma)}\,t^{\sigma-1},\quad t>0.
$$

\begin{lemma}\label{convolution}
For $\sigma>0$ and $z\in S_{\pi/4}$, we have $B^{\sigma,z}=h^\sigma\ast b^{\sigma,z}$; namely,
$$B^{\sigma,z}(s)={1\over\Gamma(\sigma)}\int_0^s(s-t)^{\sigma-1}b^{\sigma,z}(t)\,dt,\quad s>0.$$
\end{lemma}

\begin{proof} For positive $\sigma$, $z$ and $s$ we have
\begin{align*}
\frac{1}{4^\sigma\Gamma(\sigma)}\int_0^s(s-t)^{\sigma-1}&e^{-z^2/(4t)}\,\frac{dt}{t^{1+\sigma}}
=\frac{1}{4^\sigma\Gamma(\sigma)}\int_0^s\left({s\over t}-1\right)^{\sigma-1}e^{-z^2/(4t)}\,\frac{dt}{t^2} \\
&
= s^{\sigma-1}\,e^{-z^2/(4s)}\int_0^\infty r^\sigma e^{-z^2r}\,\frac{dr}{r}
=\frac{s^{\sigma-1}\,e^{-z^2/(4s)}}{z^{2\sigma}},
\end{align*}
where we have done the change of variable $4r=t^{-1}-s^{-1}$ in the second equality. From that, one gets that $B^{\sigma,z}=h^\sigma\ast b^{\sigma,z}$ for
$\sigma, z>0$. Then this equality extends to $z\in S_{\pi/4}$ by the analytic continuation principle.
\end{proof}

Clearly, $B^{\sigma,z}\not \in L^1(\R^+)$, which is not  {good} for our needs. However, the additive perturbation of $B^{\sigma,z}$ implemented by $h^\sigma$ fits very well in our setting.
Recall that $e_\varepsilon(t)=e^{-\varepsilon t}$ for $\varepsilon, t>0$.

\begin{proposition}\label{funny} Take $0<\sigma <1$, $z\in S_{\pi/4}$.
\begin{itemize}
\item[{(i)}]
The function $B^{\sigma,z}-h^\sigma$, which is given by
$$
(B^{\sigma,z}-h^\sigma)(t)
={1\over\Gamma(\sigma)}\frac{e^{-z^2/(4t)}-1}{t^{1-\sigma}},\quad \hbox{for}~t > 0,
$$
 belongs to $\mathcal{T}^{(N)}(t^{N})$ for every $N\in\N$, and
$$
\norm{B^{\sigma,z}-h^\sigma}_{(N)}
\le C_{N,\sigma}{\vert z\vert^{2N}\over(\Re(z^2))^{N-\sigma}},\quad z\in S_{\pi/4}.
$$
As a consequence, $(B^{\sigma,z}-h^\sigma)_{z\in S_{\pi/4}}
\subseteq\mathcal{T}^{(\alpha)}(t^{\alpha})$, for every $\alpha >0$.
\item[{(ii)}] For every $\varepsilon>0$ and $\alpha>0$, we have that
$(B^{\sigma,z}-h^\sigma)\,e_\varepsilon$ is in $\mathcal{T}^{(\alpha)}(t^{\alpha})$. Moreover, in the norm of $\mathcal{T}^{(\alpha)}(t^{\alpha})$,
$$
\lim_{\varepsilon\to0^+}(B^{\sigma,z}-h^\sigma)\,e_\varepsilon
=(B^{\sigma,z}-h^\sigma).
$$
\end{itemize}
\end{proposition}

\begin{proof}
(i) The expression of the function $B^{\sigma,z}-h^\sigma$ follows rightly from the definitions. As regards the estimate of the statement, it is readily obtained using Lemma \ref{Lem:b grande} (i). Item (ii) follows as in the proof of Proposition \ref{Prop:b chica}. We skip the details.
\end{proof}

\section{Fractional powers of generators of integrated semigroups}\label{Section:Semigroups}

Let $\al>0$. The notion of $\al$-times integrated semigroup was introduced by M. Hieber in \cite{Hieber-Forum}, 
see also \cite[Section 3.17]{Arendt-Batty-Hieber-Neubrander}. Here we consider  $\alpha$-times integrated 
semigroups $(T_{\al}(t))_{t\geq 0}$ which are tempered. Let us recall the corresponding definition 
given in the Introduction, which is implemented by formula \eqref{resolv} relating an $\alpha$-times integrated semigroup with its generator.
Each uniformly bounded $C_0$ semigroup $(T(t))_{t\ge0}$ is a tempered $\alpha$-times integrated semigroup for
$\alpha=0$.
Other
(non-trivial) examples of integrated semigroups have been given in Section \ref{Section:Applications}.
We will make usage of the following formulas or properties of integrated semigroups.

$\bullet$ For $f\in X$ and $\beta>\al$ define
\begin{equation}\label{intintegrated}
T_{\beta}(t)f:={1\over \Gamma(\beta-\al)}\int_0^t  (t-s)^{\beta-\al-1}T_\alpha(s)f\,ds.
\end{equation}
Then $T_{\beta}(t)$ is a globally tempered $\beta$-times integrated semigroup in ${\mathcal B}(X)$ with generator $A$ as well.

$\bullet$ For $f\in{\mathcal D}(A)$ (see \cite[Proposition~2.4]{Hieber-Forum})
\begin{equation}\label{integra}
T_{\alpha}(t)f-{t^\alpha\over\Gamma(\alpha+1)}\,f=\int_0^tT_\alpha(s)Af\,ds = T_{\alpha +1}(t)Af,\quad t\ge0.
\end{equation}

$\bullet$ Clearly from the definition, the generator $A$ of a tempered $\al$-times semigroup satisfies
\begin{equation}\label{nonnegative}
\norm{\lambda(\lambda-A)^{-1}}\le C, \quad \hbox{for all}~\lambda>0,
\end{equation}
and then, for $0<\sigma<1$, the fractional power $(-A)^\sigma$ (in the Balakrishnan sense) exists and is given by 
(see \cite[p.~260]{Yosida} or \cite{Martinez-Sanz-Marco})
\begin{equation} \label{fractpow}
(-A)^\sigma f={\sin(\pi\sigma)\over\pi}\int_0^\infty\lambda^{\sigma-1}
(\lambda-A)^{-1}(-Af)\,d\lambda,\quad f\in \mathcal D(A).
\end{equation}

$\bullet$ For $\varepsilon>0$ and $0<\sigma<1$, the operator $(\varepsilon-A)$ also satisfies the estimate (\ref{nonnegative}) and then
(see \cite[Theorem~2.1~and~Theorem~3.2]{Martinez-Sanz-Marco})
\begin{equation}\label{msm}
\lim_{\varepsilon\to0^+}(\varepsilon-A)^{-\sigma}(-A)^\sigma f
=f, \quad f\in {\mathcal D}(A).
\end{equation}

$\bullet$ There is a useful connection between the algebras $\mathcal{T}^{(\alpha)}(t^{\alpha})$, $\al\ge0$, {(introduced in the previous section)} and tempered integrated semigroups $T_\al(t)$. The mapping $\pi_\al:\mathcal{T}^{(\alpha)}(t^{\alpha})\to{\mathcal B}(X)$ defined by
\begin{equation}\label{intemor}
\pi_\al(\varphi)f:=\int_0^\infty W^{\al}\varphi(t)\,T_{\al}(t)f\,dt,\quad f\in X,~\varphi\in\mathcal{T}^{(\alpha)}(t^{\alpha}),
\end{equation}
is a bounded Banach algebra homomorphism. Also,
\begin{equation}\label{cero}
-A\pi_\al(\varphi)f=\pi_\al(\varphi')f+\varphi(0)f,
\end{equation}
for every $\varphi\in C^1[0,\infty)$ such that $\varphi'\in\mathcal{T}^{(\alpha)}(t^{\alpha})$.
See \cite[Theorem~3.1]{Miana-Forum}, where this is shown for $\varphi$ in the Schwartz class. 
The proof works also in the case above.

We now give an extension, for general tempered integrated semigroups, of a well-known formula for $C_0$-semigroups; see \cite[p.~260]{Yosida}.

\begin{theorem}\label{Thm:fractional power}
Let $A$ be the generator of a tempered $\alpha$-times integrated semigroup
$(T_\alpha(t))_{t\ge 0}$ and let $\sigma$ be such that $0<\sigma<1$.
Then, for  $f\in {\mathcal D}(A)$,
$$
(-A)^\sigma f={\Gamma(\sigma+\alpha+1)\over\Gamma(-\sigma)\Gamma(1+\sigma)}\int_0^\infty\left(T_\alpha(t)f-{t^{\alpha}\over\Gamma(\alpha+1)}\,f\right)\,{dt\over t^{\sigma +\alpha+1}}.
$$
\end{theorem}

\begin{proof}
Set $\displaystyle T_{\alpha+1}(t)f:=\int_0^tT_\alpha(s)f\,ds$, $f\in X$, as above.
Take $f\in {\mathcal D}(A)$. Since $T_{\alpha+1}(t)$ is globally tempered with exponent $\alpha+1$, \eqref{resolv} applies  so we have
\begin{equation}\label{resolete}
(\lambda-A)^{-1}f=\lambda^{\al+1}\int_0^\infty e^{-\lambda t}T_{\al+1}(t)f\,dt,\quad \Re\lambda>0.
\end{equation}
Plugging \eqref{resolete} into \eqref{fractpow}
we see that the order of integration can be exchanged, since
$\norm{T_{\alpha+1}(t)Af}\le C\norm{Af}t^{\al+1}$, if  $0\le t\le1$,
and
$\norm{T_{\alpha+1}(t)Af}\le C\norm{f}t^{\al}$, if $t\ge1$,
by \eqref{integra}. Thus we obtain
\begin{align*}
    (-A)^\sigma f &= {\sin(\pi\sigma)\over\pi}\int_0^\infty\lambda^{\sigma+\alpha}
    \int_0^\infty e^{-\lambda t}T_{\alpha+1}(t)(-Af)\,dt\,d\lambda \\
     &= {\sin(\pi\sigma)\over\pi}\int_0^\infty\lambda^{\sigma+\alpha}\int_0^\infty e^{-\lambda t}\left({t^\alpha\over\Gamma(\alpha+1)}\,f-T_\alpha(t)f\right)dt\,d\lambda \\
     &= {\Gamma(\sigma+\alpha+1)\over\Gamma(-\sigma)\Gamma(1+\sigma)}\int_0^\infty \left(T_\alpha(t)f-{t^\alpha\over\Gamma(\alpha+1)}\,f\right)\frac{dt}{t^{\sigma+\alpha+1}},
\end{align*}
as we wanted to show.
\end{proof}

\section{Proof of Theorem \ref{extension}}\label{Section:Proof}

\begin{proof}[Proof of Theorem \ref{extension}]
First we show that a solution to (\ref{ExtSemig}) is given by formula {(\ref{solution}), that is to say}, by $u(z)=\pi_\al(b^{\sigma,z})f$, 
for $f\in X$ and $z\in S_{\pi/4}$.
In fact, the equation in Proposition \ref{Prop:b chica} (ii) holds in $\mathcal T^{(\al)}(t^\al)$, whence we {immediately} obtain
$$
{\partial^2\over\partial z^2}\,\pi_\al(b^{\sigma,z})f
+{1-2\sigma\over z}\,{\partial\over\partial z}\,\pi_\al(b^{\sigma,z})f
=\pi_\al\left({\partial\over\partial t}\,b^{\sigma,z}\right)f,
$$
for every $z\in S_{\pi/4}$ and $f\in X$. On the other hand, one has
$$
\pi_\al\left({\partial\over\partial t}\,b^{\sigma,z}\right)f=-A\pi_\al(b^{\sigma,z})f
$$
by (\ref{cero}), so $u(z)=\pi_\al(b^{\sigma,z})f$ satisfies the
differential equation of (\ref{ExtSemig}). Next, we show that
$\pi_\al(b^{\sigma,z})f$ satisfies the boundary condition in (\ref{ExtSemig}).

By using the identity
$\displaystyle\frac{1}{\Gamma(\al)}\int_0^t(t-s)^{\al-1}\,ds=\frac{t^\al}{\Gamma(\al+1)}$ and Fubini's theorem one gets
$$
\int_0^\infty W^\al b^{\sigma,z}(t)\,{t^\al\over\Gamma(\al+1)}~dt
=\int_0^\infty b^{\sigma,z}(t)~dt= 1, \quad z\in S_{\pi/4}.
$$
Take for a moment $f\in{\mathcal D}(A)$. By applying the identity above and \eqref{integra} with
$n\ge\alpha$ one obtains
\begin{align*}
u(z)-f &= \int_0^\infty W^n(b^{\sigma,z})(t)
\left(T_n(t)-{t^n\over\Gamma(n+1)}\right)f\,dt \\
&= \int_0^\infty(-1)^n(b^{\sigma,z})^{(n)}(t)T_{n+1}(t)Af\,dt,
\end{align*}
where
$\displaystyle T_{n+1}(t)
=\int_0^\infty(t-s)^{n-\al-1}T_{\al+1}(s)Af\,{ds\over\Gamma(n-\al)}$.

Put $\delta=\Re(z^2)$. By Proposition \ref{Prop:b chica}(i) and Lemma \ref{Lem:b chica}(i), one gets
\begin{align*}
\|u&(z)-f\|
\le \int_0^{\sqrt\delta} \vert (b^{\sigma,z})^{(n)}(t)\vert\, \Vert T_{n+1}(t)Af\Vert\, dt\\
&
\qquad\qquad\quad+\int_{\sqrt\delta}^\infty \vert (b^{\sigma,z})^{(n)}(t)\vert\,\Vert T_{n}(t)-{t^n\over n!}\Vert\, \Vert f\Vert\,dt\\
&\le
C \Vert b^{\sigma,z}\Vert_{(n)}\Vert Af\Vert\,\sqrt\delta\\
&\quad+
(C+(1/n!))\Vert f\Vert
\sum_{j=0}^n
{\vert d_{j,n}^\sigma\vert\over 4^{j+\sigma}\Gamma(\sigma)}\,
\vert z\vert^{2(j+\sigma)}
\int_{\sqrt\delta}^\infty{e^{-\Re z^2/4t}\over t^{1+\sigma+j}}\,dt \\
&\le C_{\sigma,f}\left({\vert z^2\vert\over\Re(z^2)}\right)^{n+\sigma}\sqrt\delta 
+\sum_{j=0}^n C_{j,n}\left({\vert z^2\vert\over\Re(z^2)}\right)^{\sigma+j}
\int_{1/\sqrt\delta}^\infty {e^{-1/4s}\over s^{1+\sigma+j}}\,ds.
\end{align*}
Whence, for $f$ in ${\mathcal D}(A)$, $\lim_{z\to0} u(z)=f$ uniformly in $S_{\pi/4-\eta}$ for each fixed $\eta$ such that $0\le \eta<\pi/4$. Since
$\displaystyle\Vert u(z)\Vert=\Vert\pi_\alpha(b^{\sigma,z})f\Vert\le C_{n,\sigma}
\left({\vert z^2\vert\over\Re(z^2)}\right)^{n+\sigma}
\Vert f\Vert$ for all $f\in X$ and ${\mathcal D}(A)$ is dense in $X$ one obtains that $\lim_{z\to0} u(z)=f$ uniformly in $S_{\pi/4-\eta}$, for every
$f\in X$. This proves the first part of the theorem.

We now prove the second equality of (\ref{zwei}).
Take $n=[\al]+1$. Since $T_\al(t)$ is a tempered $\al$-times integrated semigroup the family $T_n(t)$ given by
$$
\displaystyle T_n(t)f:=\Gamma(n-\al)^{-1}\int_0^t (s-t)^{n-\al-1}T_\al(s)f\,ds, \quad f\in X,
$$
is a tempered $n$-times integrated semigroup as well. Hence, for $f\in X$ and $z\in S_{\pi/4}$,
$$
u(z)=(-1)^n\frac{z^{2\sigma}}{4^\sigma\Gamma(\sigma)}
\int_0^\infty {d^n\over dt^n}\left(\frac{e^{-z^2/(4t)}}{t^{1+\sigma}}\right)T_n(t)f\,dt,
$$
and therefore
\begin{equation}\label{solucionPrima}
\,u'(z)=\frac{(-1)^n\,z^{2\sigma-1}}{4^\sigma\Gamma(\sigma)}
\int_0^\infty {d^n\over dt^n}
\left(\Big(2\sigma-{z^2\over 2t}\Big)\,{e^{-z^2/(4t)}\over t^{1+\sigma}}\right)\,T_n(t)f\,dt.
\end{equation}

On the other hand,
$$\int_0^\infty {d^n\over dt^n}
\left(\Big(2\sigma-{z^2\over 2t}\Big)\,{e^{-z^2/(4t)}\over t^{1+\sigma}}\right)\,
{t^n\over n!}\,dt =\int_0^\infty
\Big(2\sigma-{z^2\over 2t}\Big)\,{e^{-z^2/(4t)}\over t^{1+\sigma}}\,dt,
$$
and the second integral above vanishes (this can be seen by making the change of variable
$r=z^2/4t$ for $z>0$; then the integral is null for every $z\in S_{\pi/4}$ by the analytic continuation principle).

According to the two preceding remarks we have, for every $f\in X$,
$$
z^{1-2\sigma}u'(z)=
{(-1)^n\over 4^\sigma\Gamma(\sigma)}
\int_0^\infty {d^n\over dt^n}
\left(\Big(2\sigma-{z^2\over 2t}\Big)\,{e^{-z^2/(4t)}\over t^{1+\sigma}}\right)
\Big(T_n(t)f-{t^n\over n!}f\Big)\,dt
$$
which, once having used the product derivation rule, gives us
\begin{align*}
z^{1-2\sigma}u'(z)
&=
\sum_{k=0}^n\sum_{m=0}^k D_{m,k,n}^\sigma\,
a^{m+1}\int_0^\infty{e^{a/t}\over t^{\sigma+m+n+2}}\,
\Big(T_n(t)f-{t^n\over n!}f\Big)\,dt \\
&\quad+\sum_{m=0}^n C_{m,n,n}^\sigma\,
a^m\int_0^\infty{e^{a/t}\over t^{\sigma+m+n+1}}\,
\Big(T_n(t)f-{t^n\over n!}f\Big)\,dt,
\end{align*}
with $a=-z^2/4$ and 
$$
C_{0,n,n}^\sigma
={2\sigma(-1)^n\over 4^\sigma\Gamma(\sigma)}\, d_{0,n}^\sigma
={2\sigma\over 4^\sigma\Gamma(\sigma)}\, c_{0,n,\sigma}
={2\sigma\over 4^\sigma\Gamma(\sigma)}\
{\Gamma(\sigma+n+1)\over \Gamma(\sigma+1)}\, ,
$$
where $d_{0,n}^\sigma$ and $c_{0,n,\sigma}$ are given by (\ref{EatDes}) and (\ref{Eat}) respectively.

Take now $f\in{\mathcal D}(A)$. We are going to see that the additive terms in the expression above of $z^{1-2\sigma}u'(z)$ tend to $0$, as $a\to0$, with the only exception of the term with constant $C_{0,n,n}^\sigma$. 
This last term will allow us to obtain the fractional power $(-A)^\sigma$.
Actually, for $1\le j\le n+1$,
\begin{align*}
&\vert a\vert^j \int_0^\infty{e^{(\Re a)/t}\over t^{\sigma+j+n+1}}\,
\Big\Vert T_n(t)f-{t^n\over n!}f\Big\Vert\,dt\\
&\le
\vert a\vert^j \int_0^\infty{e^{(\Re a)/t}\over t^{\sigma+j+n+1}}\,
\Vert T_{n+1}(t)\Vert \Vert Af\Vert\, dt\\
&\le C \vert a\vert^j
\int_0^\infty{e^{(\Re a)/t}\over t^{\sigma+j}}\,dt
=C\left({\vert a\vert\over\Re a}\right)^j
\left(\int_0^\infty{e^{1/s}\over s^{\sigma+j}}\,ds\right)
(\Re a)^{1-\sigma},
\end{align*}
where we have applied (\ref{integra}) in the first inequality.
Moreover,
$$
\lim_{a\to0}\int_0^\infty{e^{a/t}\over t^{\sigma+n+1}}\,\Big(T_n(t)f-{t^n\over n!}f\Big)\,dt
=\int_0^\infty{(T_n(t)f-(t^n/n!)f)\over t^{\sigma+n+1}}\,dt
$$
by the Dominated Convergence Theorem (use again (\ref{integra}) near the origin).
Then, making $z$ converging to $0$ through subsectors of $S_{\pi/4}$, and applying Theorem 4.1, we obtain
\begin{align*}
\lim_{z\to0} z^{1-2\sigma}u'(z)
&={2\sigma\over4^\sigma\Gamma(\sigma)}
{\Gamma(\sigma+n+1)\over\Gamma(\sigma+1)}
{\Gamma(-\sigma)\Gamma(\sigma+1)\over\Gamma(\sigma+n+1)}(-A)^\sigma f \\
&={2\sigma\Gamma(-\sigma)\over4^\sigma\Gamma(\sigma)}(-A)^\sigma f.
\end{align*}

Next, we are going to see that the solution $u(z)$ above satisfies formula (\ref{ein}). By Lemma \ref{convolution}, $B^{\sigma,z}=b^{\sigma,z}\ast h^\sigma$, $z\in S_{\pi/4}$. Thus we have $B^{\sigma,z} e_\varepsilon=(b^{\sigma,z} e_\varepsilon)\ast(h^\sigma e_\varepsilon)$ in $\mathcal T^{(\al)}(t^\al)$ for every $\varepsilon>0$. Moreover,
\begin{equation}\label{unoss}
(\varepsilon -A)^{-\sigma}=\pi_\al(e_\varepsilon h^\sigma),\quad \varepsilon>0;
\end{equation}
see \cite[Remark~4.10]{Miana-Forum}.
Then, by Proposition \ref{Prop:b chica}(iii), \eqref{msm}, \eqref{unoss} and
Lemma \ref{convolution}, we get for $f\in\mathcal D(A)$ that
\begin{align*}
u(z) &= \pi_\al(b^{\sigma,z})f=\lim_{\varepsilon\to0^+}\pi_\al(b^{\sigma,z} e_\varepsilon)(\varepsilon-A)^{-\sigma}(-A)^\sigma f \\
 &= \lim_{\varepsilon\to0^+}\pi_\al(b^{\sigma,z} e_\varepsilon)\pi_\al(h^\sigma e_\varepsilon)(-A)^\sigma f
 =\lim_{\varepsilon\to0^+}\pi_\al(B^{\sigma,z} e_\varepsilon)(-A)^\sigma f,
\end{align*}
uniformly in $z$ running over compact subsets of $S_{\pi/4}$. Thus \eqref{ein} is proved.

Formula (\ref{solfract}) is rightly obtained from the above. In fact, that formula
can be written as
\begin{equation}\label{UcomoB}
u(z)-f=\pi_\al(B^{\sigma,z}-h^\sigma)(-A)^\sigma f.
\end{equation}
To see that (\ref{UcomoB}) holds one only needs to apply the limit above and Proposition \ref{funny} (ii):
\begin{align*}
u(z)&-f=\lim_{\varepsilon\to0^+}\left(\pi_\al(B^{\sigma,z}e_\varepsilon)(-A)^\sigma f
-\pi_\al(h^\sigma e_\varepsilon)(-A)^\sigma f\right)\\
&=\lim_{\varepsilon\to0^+}
\pi_\al\left((B^{\sigma,z}-h^\sigma)e_\varepsilon\right)(-A)^\sigma f 
=\pi_\al(B^{\sigma,z}-h^\sigma)(-A)^\sigma f.
\end{align*}

It remains to prove the characterization of the fractional power $(-A)^\sigma$ which is indicated in the first equality of formula (\ref{zwei}).
Notice that
\begin{align*}
&\frac{z^{-2\sigma}}{\Gamma(\sigma)}
\int_0^\infty{d^n\over dt^n}\left(\frac{e^{-z^2/(4t)}-1}{t^{1-\sigma}}\right){t^n\over n!}\,dt
=\frac{z^{-2\sigma}}{\Gamma(\sigma)}
\int_0^\infty\frac{e^{-z^2/(4t)}-1}{t^{1-\sigma}}\,dt\\
&= {4^{-\sigma}\over\Gamma(\sigma)}\int_0^\infty r^{\sigma-1}(e^{-1/r}-1)\,dr
    =-{4^{-\sigma}\over\Gamma(\sigma)}\int_0^\infty\int_0^s e^{-u}\,du\,s^{-(1+\sigma)}\,ds \\
     &= -{4^{-\sigma}\over\sigma\Gamma(\sigma)}\int_0^\infty u^{-\sigma}e^{-u}\,du
    ={\Gamma(-\sigma)\over4^{\sigma}\Gamma(\sigma)}=c_\sigma.
\end{align*}
Hence,
\begin{align*}
\displaystyle z^{-2\sigma}&(u(z)-f)-c_\sigma(-A)^\sigma f\\
&=\frac{z^{-2\sigma}}{\Gamma(\sigma)}
\int_0^\infty{d^n\over dt^n}\left(\frac{e^{-z^2/(4t)}-1}{t^{1-\sigma}}\right)\left(T_n(t)-{t^n\over n!}\right)(-A)^\sigma f\,dt\\
&=\frac{z^{-2\sigma}}{\Gamma(\sigma)}
\int_0^{\sqrt{\Re (z^2)}}
{d^n\over dt^n}\left(\frac{e^{-z^2/(4t)}-1}{t^{1-\sigma}}\right)T_{n+1}(t)A(-A)^\sigma f\,dt\\
&+
\frac{z^{-2\sigma}}{\Gamma(\sigma)}\int_{\sqrt{\Re (z^2)}}^\infty
{d^n\over dt^n}\left(\frac{e^{-z^2/(4t)}-1}{t^{1-\sigma}}\right)\left(T_n(t)-{t^n\over n!}\right)(-A)^\sigma f\,dt.
\end{align*}
Thus
\begin{align*}
\displaystyle\Vert &z^{-2\sigma}(u(z)-f)-c_\sigma(-A)^\sigma f\Vert\\
&\le C
\frac{\vert z\vert^{-2\sigma}}{\Gamma(\sigma)}\,\sqrt{\Re (z^2)}\,\Vert A(-A)^\sigma f\Vert\,
\int_0^{\sqrt{\Re (z^2)}}\Big\vert{d^n\over dt^n}\left(\frac{e^{-z^2/(4t)}-1}{t^{1-\sigma}}\right)\Big\vert t^ndt\\
&\quad+C\frac{\vert z\vert^{-2\sigma}}{\Gamma(\sigma)}\int_{\sqrt{\Re (z^2)}}^\infty
\Big\vert{d^n\over dt^n}\left(\frac{e^{-z^2/(4t)}-1}{t^{1-\sigma}}\right)\Big\vert\, t^n\, dt.
\end{align*}
Then the first term in the right hand of the inequality is bounded by
\begin{align*}
C\vert z\vert^{-2\sigma}\sqrt{\Re (z^2)}\,
\Vert B^{\sigma,z}-h^\sigma\Vert_{(n)}
&\le
C_{n,\sigma}
\left({\vert z^2\vert \over \Re (z^2)}\right)^n\,\sqrt{\Re (z^2)},
\end{align*}
by Proposition \ref{funny} (i), 
whereas the second member is bounded by
\begin{align*}
C\vert z\vert^{-2\sigma}&\int_{\sqrt{\Re (z^2)}}^\infty
\left(\frac{1-e^{-\Re (z^2)/(4t)}}{t^{1-\sigma}}\right)\,dt\\
&\quad+
\sum_{k=1}^n\sum_{m=1}^k C_{m,k}
\vert z\vert^{-2\sigma}\Big\vert -{z^2\over 4}\Big\vert^m
\int_{\sqrt{\Re (z^2)}}^\infty   \frac{e^{-\Re (z^2)/(4t)}}{t^{m+1-\sigma}}\,dt\\
&=
C\left({\Re (z^2)\over \vert z^2\vert}\right)^{\sigma}
\int_{1/\sqrt{\Re (z^2)}}^\infty
\left(\frac{1-e^{-1/(4s)}}{s^{1-\sigma}}\right)\,ds\\
&\quad+
\sum_{k=1}^n\sum_{m=1}^k C_{m,k}
\left({\vert z^2\vert\over \Re (z^2)}\right)^{m-\sigma}
\int_{1/\sqrt{\Re (z^2)}}^\infty
\frac{e^{-1/(4s)}}{s^{m+1-\sigma}}\,ds.
\end{align*}
Putting all the preceding estimates together one obtains that, as it was claimed,
$\lim_{z\to0} z^{-2\sigma}(u(z)-f)=c_\sigma(-A)^\sigma f$, $f\in{\mathcal D}(A)$, through subsectors of $S_{\pi/4}$.
We have proved all the formulas of the theorem.

Finally, the boundedness of the solution $u$ on subsectors of $S_{\pi/4}$ is easy to show. Take $n>\al$. Since $A$ is also the generator of the integrated semigroup $T_n(t)$ defined in (\ref{intintegrated}), it follows by \eqref{solution} applied to $T_n(t)$ and Proposition \ref{Prop:b chica} (i) that
\begin{align*}
\norm{u(z)}&\leq C_n\frac{\abs{z^{2\sigma}}\Vert f\Vert}{4^\sigma\Gamma(\sigma)}
\int_0^\infty
{d^n\over dt^n}\left({e^{-z^2/(4t)}\over t^{1+\sigma}}\right)t^n\,dt
\le C_{n,\sigma}
\left({\vert z^2\vert\over\Re(z^2)}\right)^{n+\sigma} \Vert f\Vert,
\end{align*}
for every $z\in S_{\pi/4}$, as we wanted to show.
\end{proof}

\begin{section} {Appendix: Fractional complex parameter}

In the present paper, and so in Theorem \ref{extension} particularly, we have focused on positive powers $(-A)^\sigma$, $0<\sigma<1$, 
of the operator $-A$ because this is the case originally dealt with in \cite{Caffarelli-Silvestre} looking for applications to PDE's. 
See also \cite{Banica-Gonzalez-Saez, Caffarelli-Salsa-Silvestre, Caffarelli-Vasseur,
Chang-Gonzalez, Ferrari-Franchi, RoncalStinga, Stinga-Torrea, StingaZhang}. 
It seems on the other hand to be worth saying something about solutions to (\ref{ExtSemig}) for complex $\sigma$ such that 
$\Re \sigma>0$.

Thus, note first that the definiton (\ref{fractpow}) of $(-A)^\sigma$ is also valid for $0<\Re\sigma<1$ and then we get defined 
$(-A)^\sigma$ for all $\sigma\in\C^+$ on suitable $f$. Moreover, the function $b^{\sigma,z}$ given by (\ref{bdef}) is also defined for $z\in S_{\pi/4}$ and 
$\Re\sigma>0$, and belongs to the Sobolev algebra ${\mathcal T}^N(t^N)$ for every $N$. In fact, for 
$z\in S_{\pi/4}$ and $\lambda\in\C$ we have 
$\vert z^\lambda\vert\leq e^{\vert\Im \lambda\vert(\pi/4)}\,\vert z\vert^{\Re\lambda}=C_\lambda\,\vert z\vert^{\Re\lambda}$ and therefore the estimate obtained in Proposition (\ref{Prop:b chica})(i) remains true in the form
$$
\norm{b^{\sigma,z}}_{(N)}\le C_{N,\sigma}\left(\abs{z}^2\over\Re(z^2)\right)^{N+\Re\sigma},
\quad z\in S_{\pi/4},\, \Re\sigma>0.
$$

Then the arguments used in the first part of the proof in Section \ref{Section:Proof}
below work to show that the function $u$ given by the integral of (\ref{solution})
is a solution of problem (\ref{ExtSemig}) for every $\sigma\in\C^+$, and that the formula
$$
(-A)^\sigma f={1\over2\sigma\, c_\sigma}\lim_{z\to0}\, z^{1-2\sigma} u'(z)
$$
of (\ref{zwei}) holds true whenever $0<\Re\sigma<1$. In case $\Re\sigma>0$ and $\sigma$ noninteger, 
it can be proven by induction that
for $f$ in the domain of $(-A)^n$, where $n$ is the integer part of $\Re\sigma$,
$$
(-A)^\sigma f=\mu_\sigma\lim_{z\to0}z^{2(n-\sigma)+2}\left({1\over z}{d\over dz}\right)^{n+1}u_\sigma,
$$
for a certain explicit constant $\mu_\sigma$. This extends to generators of integrated semigroups the results of L. Roncal and P. R. Stinga valid
for nonnegative selfadjoint operators with dense domains in $L^2$, see \cite{RoncalStinga}.

\end{section}

 \medskip 
\noindent{\small\textbf{Acknowledgments.} 
The third author thanks the Departamento de Ma\-te\-m\'a\-ti\-cas and I.U.M.A. at Universidad 
de Zaragoza, Spain, for its kind hospitality during several visits. The authors wish also to thank the referee for helpful comments
and a question which motivated the Appendix section.



\end{document}